\documentclass[11pt]{amsart}
\usepackage{amsmath}
\usepackage[utf8]{inputenc}
\usepackage{amssymb}
\usepackage{amsopn}
\usepackage{epsfig}
\usepackage{amsfonts}
\usepackage{latexsym}
\usepackage{graphicx}
\usepackage{enumerate}
\usepackage{color}
\usepackage{tikz, subfigure}
\newtheorem{theorem}{Theorem}[section]

\newtheorem{proposition}[theorem]{Proposition}
\newtheorem{corollary}[theorem]{Corollary}

\theoremstyle{definition}

\newtheorem{conjecture}[theorem]{Conjecture}
\theoremstyle{remark}
\newtheorem{remark}[theorem]{Remark}

\numberwithin{equation}{section}

\begin{document}
\title{ On some fractal sets only containing irrationals}
\author[K. Jiang]{Kan Jiang}
\address[K. Jiang]{Department of Mathematics, Ningbo University,
People's Republic of China}
\email{jiangkan@nbu.edu.cn}
\date{\today}
\begin{abstract}
In this paper, we  prove that many fractal sets  generated by the associated  dynamical systems only contain irrationals. We give two applications. Firstly,  we explicitly construct some overlapping self-similar sets which  only consist of irrationals. Secondly, we prove that  some scaled   univoque sets only contain irrationals.
\end{abstract}
\maketitle
\section{Introduction}
Kurt Mahler \cite{Mahler1984} posed the following famous conjecture:
\begin{conjecture}
Any number from the ternary Cantor set is either a rational or a transcendental number.
\end{conjecture}
This conjecture is elegant. It is a classical problem in  Diophantine  approximation.   To the best of our knowledge, this conjecture is still far from being solved. However, it motivates us to consider the numbers in fractal sets.  For instance, it is natural to ask when some fractal sets only consist of irrationals or transcendental numbers.   We introduce some related work on this topic.
Jia, Li and Jiang \cite{Jia2021} proved the following result:
\begin{proposition}
Let $E\subset \mathbb{R}$ be a set  with zero  Lebesgue  measure. Then for  Lebesgue almost every $t$, we have that
$$E+t=\{x+t:x\in E\} $$ is a set only containing    irrationals or    transcendental numbers.
\end{proposition}
Moreover, they considered some special Cantor sets, and constructed explicitly $t$'s such that $E+t$ only contains irrationals.
In this paper, we shall analyze some more general fractal sets, including invariant sets of some dynamical systems, self-similar sets  with overlaps and so forth.
We prove that these fractal sets only contain irrational  numbers.
Before  introducing  the main results of this paper, we give  some definitions and notation.
 Let $m\in \mathbb{N}_{\geq 2}$ and  $t\in (0,1)$. We say that  $t$ has an $m$-adic expansion  if there exists some $(t_k)\in\{0,1,2,\cdots,m-1\}^{\mathbb{N}}$ such that
$$t=\sum_{k=1}^{\infty}\dfrac{t_k}{m^k}.$$  Given any non-empty word $(x_1x_2\cdots x_k)\in \{0,1,2,\cdots,m-1\}^{k},k\geq 1.$ If there exists some $k_0\in\mathbb{N}$ such that
$$t_{k_0+1}t_{k_0+2}\cdots t_{k_0+k}=x_1x_2\cdots x_k,$$ then we call  $(t_k) $  a universal expansion of $t$.

We say that   an expansion $(t_k)$ is normal if
 for any finite nonempty word $\textbf{a}$  taken from $$\{0,1,2,\cdots, m-1\}^{*}=\cup_{k=1}^{\infty} \{0,1,2,\cdots, m-1\}^{k}$$ we have  that
 the limiting frequency of the appearances of  $\textbf{a}$ as a subword of $(t_k)$  is $$\dfrac{1}{m^{|\textbf{a}|}},$$ where $|\textbf{a}|$ denotes the length of $\textbf{a}.$ If $x$ has a normal expansion, then we call $x$ a normal number.
 We define the following sets.
$$A\pm t=\{x\pm t:x\in A\}, tA=\{tx: x\in A\setminus \{0\}\},  t^{-1}A=\{t^{-1}x: x\in A\setminus \{0\}\}.$$
In what follows, we will use  similar notation.
 Now we state the main result of this paper.
\begin{theorem}\label{Main}
Let $m\in \mathbb{N}_{\geq 2}.$ We define
$$T:[0,1)\to [0,1)$$
by $$T(x)=mx\mod 1.$$
Let $A$ be  any  nowhere dense invariant set of $T$.  For
 any $t\in (0,1)$ such that
$t$ has a universal $m$-adic expansion,  then we simultaneously  have
$$A-t \subset  \mathbb{Q}^c, A+t\subset \mathbb{Q}^c, \dfrac{A}{t}\subset  \mathbb{Q}^c.$$
Moreover, if $t\in (1,\infty)$ and   $(1/t)$ has a universal $m$-adic expansion, then
$$tA\subset \mathbb{Q}^c.$$
In particular, for  Lebesgue almost every $t\in (0,1)$, we have the above inclusions.
\end{theorem}
\begin{remark}
A set $A$ is called an invariant set of $T$ if $T(A)\subset A.$
There are many invariant sets with respect to $T$.
 Therefore, our result gives a uniform $t$ such that
 $A\pm t, A/t$
  and
  $tA$ only contain irrationals.
  Dynamically, for any $t\in (0,1)$ with an $m$-adic universal  expansion, its orbit
$$\{T^n(t):n\geq 0\}$$ is dense in $[0,1]$. In fact, our condition can be weaken.
 We only need to assume that $(t_k)$ is ``local" universal, i.e. it suffices to assume that
 $$(t_k)=x_1x_2\cdots x_i s_1s_2s_3\cdots,$$
 where $(s_j)$ is a universal expansion and $x_1x_2\cdots x_i$ is any word from $$\{0,1,2,\cdots,m-1\}^{i} .$$ Here $i$ can be any positive integer. Roughly speaking, ``local" universal means the dynamical orbit of $t$ is locally dense in some subset of $[0,1]$. 
The last statement is trivial as it is well-known that for Lebesgue almost every $t\in (0,1)$, $t$ has an $m$-adic  universal expansion. This is due to the famous Borel's normal number theorem.
\end{remark}
If  $A$ is a special set, for instance, the survival set of some open dynamical system \cite{Glendinning-Sidorov}, then we still have the above result. For simplicity, we only consider one hole. For multiple holes, our result is still correct.
\begin{corollary}\label{cor2}
Given a hole $(a,b)\subset [0,1)$.
Let $A$ be the survival set of the following open dynamics, i.e.
$$A=\{x\in [0,1): T^n(x)\notin (a,b) \mbox{ for any }n\geq 0\}.$$  Then for any $t\in (0,1)$ such that
$t$ has a universal $m$-adic expansion, we have
$$A-t \subset  \mathbb{Q}^c,A+t\subset \mathbb{Q}^c, \dfrac{A}{t}\subset  \mathbb{Q}^c.$$
Moreover, if $t\in (1,\infty)$ and   $(1/t)$ has a universal $m$-adic expansion, then
$$tA\subset \mathbb{Q}^c.$$
\end{corollary}
With   similar discussion, we are allowed to prove the following result.
\begin{corollary}\label{cor3}
Given any $m\in \mathbb{N}_{\geq 2}$.
Let $\Sigma$ be a subshift  of finite type \cite{LM} over $\{0,1,2,\cdots, m-1\}^{\mathbb{N}}$. Let $$A=\left\{\sum_{i=1}^{\infty}\dfrac{x_i}{m^i}: (x_i)\in \Sigma\right\}.$$
If $A$ is nowhere dense in $[0,1]$, then
for any $t\in (0,1)$ such that
$t$ has a universal $m$-adic expansion, we have
$$A-t \subset  \mathbb{Q}^c,A+t\subset \mathbb{Q}^c, \dfrac{A}{t}\subset  \mathbb{Q}^c.$$
Moreover, if $t\in (1,\infty)$ and   $(1/t)$ has a universal $m$-adic expansion, then
$$tA\subset \mathbb{Q}^c.$$
\end{corollary}
In Theorem \ref{Main}, if we choose $t$ with stronger property, then we are able to prove the following result. We still use the notation defined in Theorem \ref{Main}.
\begin{theorem}\label{Main1}
Let $A$ be  any  nowhere dense invariant set of $T$. Then for  any normal number $t\in (0,1)$,
$$A+it\subset \mathbb{Q}^c, \dfrac{A}{it}\subset \mathbb{Q}^c, i\in \mathbb{Z}\setminus \{0\}.$$
\end{theorem}
In the above theorem, we let
$$A+it=\{x+it:x\in A\}, \dfrac{A}{it}=\left\{\dfrac{x}{it}:x\in A\setminus \{0\}\right\}, i\neq 0.$$
As an application of the idea of  Theorem \ref{Main}, we explicitly construct some overlapping self-similar sets  which  only consist of irrationals.
\begin{theorem}\label{Main2}
Let $K$ be the attractor of the IFS
$$\left\{f_i(x)=\dfrac{x+a_i}{q}\right\}_{i=1}^{n},$$
where
$$q\in \mathbb{N}_{\geq 3},a_i\in \mathbb{Q}\cap [0,q-1].$$  Suppose that   $K$ is nowhere dense in $[0,1]$.
Then for any $t\in (0,1)$ such that
$t$ has a universal $q$-adic expansion, we have
$$K-t\subset \mathbb{Q}^c, K+t\subset \mathbb{Q}^c, \dfrac{K}{t}\subset  \mathbb{Q}^c.$$
Moreover,  for any $t\in (1,+\infty)$ such that
$1/t$ has a universal $q$-adic expansion in $[0,1]$,  then we have
$$ tK\subset \mathbb{Q}^c. $$
\end{theorem}
The main idea of Theorem \ref{Main}
 is still useful for  some non-integer expanding maps. We give another application to
 the univoque set.
Let $\beta\in(1,2)$.  Given any $x\in [0,1/(\beta-1)]$. Then there exists an expansion $(x_n)\in \{0,1\}^{\mathbb{N}}$ such that
$$x=\sum_{n=1}^{\infty}\dfrac{x_n}{\beta^n}.$$
If $(x_n)$ is unique, then we call $x$ a univoque point.  Usually, for a generic point, it has uncountably many different expansions \cite{RSK,KarmaMartijn,MK,GS}. We denote by $U_{\beta}$ all the univoque points in base $\beta.$  It is well-known that for any $\beta\in(\beta_{*},2)$, $U_{\beta}$ is a uncountable set,  where $\beta_{*}$  is the Komornik-Loreti constant \cite{Allouche,GS,Komornik}.
Now, we state the final result of this paper.
\begin{theorem}\label{Multiplication}
Let $1<\beta<2.$
 Let  $G: [0,1)\to [0,1)$ be a map defined by $G(x)=\beta x\mod 1.$
Suppose that  $A$ is a nowhere dense invariant set of $G$.
Then for any $t\in (0,1)$ such that
$$\{\beta^kt\mod 1:k\geq 0\}$$ is dense in $[0,1)$,  we have
$$tA\subset \mathbb{Q}^c.$$
Moreover, for any $t\in (1,+\infty)$ such that
$$\{\beta^k(1/t)\mod 1:k\geq 0\}$$ is dense in $[0,1)$, then
$(1/t)A\subset \mathbb{Q}^c.$

 \noindent In particular, if $A=U_{\beta}\cap \left[\dfrac{2-\beta}{\beta-1},1\right]$ and $\beta_{*}<\beta<2$, then for some appropriate $t$'s,
  $tA\subset \mathbb{Q}^c \mbox{ or }(1/t)A\subset \mathbb{Q}^c.$
 \end{theorem}

This paper is arranged as follows.
In Section 2, we give the proofs of main results.  In Section 3, we give some problems.
\section{Proofs of main results}
\begin{proof}[Proof of Theorem \ref{Main}]
We only prove $A+t, t^{-1}A\subset \mathbb{Q}^c$. The other two inclusions can be proved analogously.
First, we show  $A+t\subset \mathbb{Q}^c$.
Suppose that there exists some $x\in A$ such that
$$x+t=r\in \mathbb{Q}.$$
Then we consider
$$\{m^k(r-x)\mod1:k\geq 0\}=\{m^kt\mod1: k\geq 0\}.$$
 Note that
the right side of the above equation is dense in $[0,1]$.
This is because $t$ has a universal expansion.
For the left side, it is not dense. Therefore, we obtain a contradiction and finish the proof.

Now, we prove the left side of the above equation is not dense.
We let $$r=\kappa_1/\kappa_2,\kappa_i\in \mathbb{Z}.$$
Therefore,
$$m^k(r-x)\mod1 =(m^k r -m^k x)\mod1.$$
Since $x\in A$, there exists an $m$-adic expansion $$(x_i)\in \{0,1,2,\cdots,m-1\}^{\mathbb{N}}$$
such that
$$x=\sum_{i=1}^{\infty}\dfrac{x_i}{m^i}.$$
Since $A$ is invariant, it follows that $T(x)\in A.$
Therefore, $$x^{\prime}:=m^kx\mod 1
=T^k(x)=\sum_{i=1}^{\infty}\dfrac{x_{k+i}}{m^i}\in A.$$
Now, we have
$$m^k(r-x)\mod1 =(m^k r -m^k x)\mod1=m^k\dfrac{ \kappa_1}{\kappa_2} -x^{\prime}\mod 1.$$
Therefore,
$$\{m^k(r-x)\mod1:k\geq0\}\subset \cup_{j=-1}^{1} \cup_{i=0}^{2\kappa_2} \left(\dfrac{i}{\kappa_2}+j-A\right)\mod 1.$$
Here $a-F=\{a-x: x\in F\},a \in \mathbb{R}.$

It is well-known that any finite union of nowhere dense sets is still nowhere dense. Hence,
$\{m^k(r-x)\mod1:k\geq 0\}$ is not dense in $[0,1)$.
We finish the proof of $A+t\subset \mathbb{Q}^c.$

Next, we prove that $ t^{-1}A\subset \mathbb{Q}^c$. We still use some notation when we prove $A+t\subset \mathbb{Q}^c.$
Suppose that  there exists some $x\in A$ such that
$$t^{-1}x=r\in \mathbb{Q}.$$ We let $$r=\dfrac{\kappa_3}{\kappa_4},$$ where $\kappa_i\in \mathbb{Z}.$
Note that $$\{m^k(r^{-1}x)\mod1:k\geq 0\}=\{m^k(t)\mod1: k\geq 0\}.$$
The right side of the above equation is dense. Now, we prove that the left side is not dense.
Notice that
\begin{equation*}
\begin{aligned}
m^k(r^{-1}x)\mod1
&=& r^{-1}m^kx\mod1
= r^{-1}\left(m^kx\right)\mod1\\
&=&r^{-1}\left(\sum_{i=1}^{k}m^{k-i}x_i+T^k(x)\right)\mod1\\
&=&r^{-1}\left(\sum_{i=1}^{k}m^{k-i}x_i+x^{\prime}\right)\mod1\\
&\subset&\left(r^{-1}A+\dfrac{\kappa_4}{\kappa_3}M_k\right)\mod1,
\end{aligned}
\end{equation*}
where $M_k:=\sum_{i=1}^{k}m^{k-i}x_i\in \mathbb{N}^{+}$.
Note that
 $$  \left\{\left(r^{-1}A+\dfrac{\kappa_4}{\kappa_3}M_k\right)\mod1:k\geq0\right\}\subset   \left(\left\{r^{-1}A+\dfrac{i}{\kappa_3}\mod 1\right\}_{i=0}^{\kappa_3}\right).$$
 We observe that for each $i$,  $$r^{-1}A+\dfrac{i}{\kappa_3}\mod 1$$ is a nowhere dense set as $A$ is nowhere dense. Therefore,
$$\{m^k(r^{-1}x)\mod1:k\geq 0\}$$ is contained in a union of some nowhere dense sets.
Hence, we have proved that
$$\{m^k(r^{-1}x)\mod1:k\geq 0\}$$ is not dense in $[0,1]$.
\end{proof}
\begin{proof}[Proof of Corollary \ref{cor2}]
We  need to check  that the survival set is invariant. By the definition of survival set, we have that  for any $x\in A$, $T(x)\in A$. Therefore, $T(A)\subset A.$
By the definition of open dynamics, the survival set $A$ is closed. Therefore, to prove $A$ is nowhere dense,   it suffices to prove that $A$ does not have interior. Otherwise, if $A$ contains some interval, we denote it by $(\delta_1,\delta_2)$. It is well-known by the ergodic theorem that for Lebesgue almost every $x\in (\delta_1,\delta_2)$, the orbit of $x$ under the map $T$ is dense in $[0,1]$. This is clearly a contradiction as the the orbits of points from the survival set never hit the hole $(a,b)$. We finish the proof.
\end{proof}
\begin{proof}[Proof of Corollary \ref{cor3}]
The proof is similar to that of Theorem \ref{Main}. We leave it to the reader.
\end{proof}

\begin{proof}[Proof of Theorem  \ref{Main1}]
We recall two well-known results on  the sequences of uniform distribution \cite{Kuipers1974}.

\textbf{Claim1:}
 Let $q\in \mathbb{N}_{\geq 2}$, $t\in (0,1)$. Then
$$\{(q^kt ):k\geq 0\}$$ is  of uniform  distribution in $[0,1]$ if and only if $t$ is a normal number in base $q$, where $(q^kt )$ denotes the fractional part of $q^kt .$

\textbf{Claim2:}
If $$\{(q^kt ):k\geq 0\}$$ is of  uniform  distribution in $[0,1]$, then for any $s\in \mathbb{Z}\setminus \{0\}$, we have
$$\{(q^kts ):k\geq 0\}$$ is of  uniform  distribution in $[0,1]$.

We only prove $$A+it\subset \mathbb{Q}^c, i\neq 0.$$
If there is some $x\in A$ such that
$$x+it=r\in \mathbb{Q}.$$ Then we consider
\begin{equation}\label{1}
m^k(r-x)\mod 1=m^k it\mod 1, k\geq 0.
\end{equation}
Since $t$ is normal, it follows by Claim 1 and Claim 2 that
$$\{m^k it:k\geq 0\}$$ is  uniformly distributed modulo $1$.  However,  the left side of equation (\ref{1}) is not dense in $[0,1]$ (the  reader may refer to the proof of  Theorem \ref{Main}). We find a contradiction.
\end{proof}
\begin{proof}[Proof of Theorem  \ref{Main2}]
We only prove  $K+t\subset \mathbb{Q}^c$. The other two cases can be proved in a similar way.
Suppose on contrary that
there exists some $x\in K$ such that
$$x+t=r\in \mathbb{Q}. $$
Then we have
$$q^k(r-x)\mod 1=q^kt\mod 1.$$
Since
$x\in K$, it follows that there exists a greedy expansion (roughly speaking, a greedy expansion is largest in the sense of lexicographical order, the reader may refer to \cite{KarmaMartijn} for more information) $(b_i)\in \{a_1,a_2,\cdots, a_n\}^{\mathbb{N}}$
such that
$$x=\sum_{k=1}^{\infty}\dfrac{b_k}{q^k}.$$
Therefore,
$$\{q^k(r-x)\mod 1:k\geq 0\}\subset  (F-K(\mod 1))\cap [0,1],$$ where
$$
F=\left(\bigcup_{t_i,k\in \mathbb{N}, 0\leq i\leq n}( q^kr- (t_1a_1+t_2a_2+\cdots t_na_n))\right),$$
and $F-K(\mod1)=\{y-x\mod1:x\in K, y\in F\}.$
However,  we can only take finite values from $F(\mod1)=\{x\mod 1:x\in F\}$. This is because
 $a_i, r\in \mathbb{Q}, 1\leq i\leq n$. We may assume
 $$r=\dfrac{C}{D}, a_i=\dfrac{B_i}{A_i}, $$ where
 $$C,D,A_i, B,B_i\in \mathbb{Z}^{+}.$$
 Note that $K\subset [0,1]$. Hence, if we take some $y\in F$ which is  bigger than $2$, then
 $$y-K(\mod 1)=y-1-K(\mod 1).$$
 In other words, when we consider $F-K(\mod1)$, we can  take only  finitely many values from $F$. More precisely,
 we may take  at most $$2
 D(\Pi_{i=1}^{n}A_i)+1$$ numbers from $F$.

 Hence,  $F-K(\mod1)$ is contained in a union of finitely many nowhere dense sets.
Therefore,  $\{q^k(r-x)\mod 1:k\geq 0\}$ is also nowhere dense. This contradicts with the fact
 $$\{q^kt\mod 1:k\geq 0\}$$ is dense in $[0,1].$
\end{proof}
\begin{proof}[Proof of Theorem \ref{Multiplication}]
We only prove $ t^{-1}A\subset \mathbb{Q}^c$.  The other one can be proved in a similar way. 
Suppose that  there exists some $x\in A$ such that
$$t^{-1}x=r\in \mathbb{Q}.$$ We let $$r=\dfrac{\kappa_5}{\kappa_6},$$ where $\kappa_i\in \mathbb{Z}.$
Note that $$\{\beta^k(r^{-1}x)\mod1:k\geq 0\}=\{\beta^kt\mod1: k\geq 0\}.$$
The right side of the above equation is dense. The left side, however, is not. Now, we prove that the left side is not dense.
Notice that
\begin{equation*}
\begin{aligned}
\beta^k(r^{-1}x)\mod1
&=& r^{-1}\beta^kx\mod1
= r^{-1}\left(\beta^kx\right)\mod1\\
&=&r^{-1}\left(\sum_{i=1}^{k}\beta^{k-i}x_i+G^k(x)\right)\mod1\\
&=&r^{-1}\left(\sum_{i=1}^{k}\beta^{k-i}x_i+x^{\prime}\right)\mod1\\
&\subset&\left(r^{-1}A+\dfrac{\kappa_6}{\kappa_5}M_k\right)\mod1,
\end{aligned}
\end{equation*}
where $x=\sum_{i=1}^{\infty}\dfrac{x_i}{\beta^i}$ has a greedy expansion $(x_i)$ in base $\beta$ (\cite{KarmaMartijn}),    and $$M_k:=\sum_{i=1}^{k}m^{k-i}x_i\in \mathbb{N}^{+}.$$
Note that
 $$  \left\{\left(r^{-1}A+\dfrac{\kappa_6}{\kappa_5}M_k\right)\mod1:k\geq0\right\}\subset   \left(\left\{r^{-1}A+\dfrac{i}{\kappa_5}\mod 1\right\}_{i=0}^{\kappa_5}\right).$$
 We observe that for each $i$,  $$r^{-1}A+\dfrac{i}{\kappa_5}\mod 1$$ is a nowhere dense set as $A$ is nowhere dense. Therefore,
$$\{\beta^k(r^{-1}x)\mod1:k\geq 0\}$$ is contained in a union of some nowhere dense sets.
Hence, we have proved that
$$\{\beta^k(r^{-1}x)\mod1:k\geq 0\}$$ is not dense in $[0,1]$.

For the univoque set, we  have the following simple properties.
First, we always have 
$$G\left(U_{\beta}\cap \left[\dfrac{2-\beta}{\beta-1},1\right]\right)\subset U_{\beta}\cap \left[\dfrac{2-\beta}{\beta-1},1\right].$$
Next,  for any $x\in U_{\beta}$, there exists some $k\in \mathbb{N}$ such that for any $i\geq k$ $$G^{i}(x)\in \left[\dfrac{2-\beta}{\beta-1},1\right].$$
In terms of the above two properties, we have the desired results for the univoque sets.
\end{proof}
\section{Some questions}
We pose  some problems.
\begin{itemize}
\item [(1)] How can we find a $t$ such that $C+t$ only consists of transcendental numbers, where $C$ is the ternary Cantor set.
\item [(2)] For  any two non-trivial  invariant sets, denoted by $A$ and $B$,  under the maps $T_2$ and $T_3$, respectively, then how can we find a  concrete uniform  $t$ such that
$$A+t, B+t\subset \mathbb{Q}^c,$$
where $T_2(x)=2x\mod 1 \mbox{ and } T_3(x)=3x\mod 1$ are defined on $[0,1)$.
\item [(3)]
In Theorem \ref{Main}, we do not know whether   for any $t\in (0,1)$ such that
$t$ has a universal $m$-adic expansion, we always have
$$ tA=\{ta:a\in A\setminus\{0\}\}\subset  \mathbb{Q}^c.$$
\end{itemize}
For the last question,
the only thing we know is the following fact from uniform distribution modulo one.
Let
$q\in \mathbb{N}_{\geq 2}$. Then for Lebesgue almost every $a$,
$$\{q^ka\mod 1:k\geq 0\}$$ is  of uniform distribution  in the unit interval.
Therefore,  for Lebesgue almost every $a$, we simultaneously have
that
$$\{q^ka\mod 1:k\geq 0\}$$
and
$$\{q^ka^{-1}\mod 1:k\geq 0\}$$ is  of uniform distribution in $[0,1]$.
By the above analysis, we  claim that for Lebesgue almost every $t\in (0,1)$,
$$A-t \subset  \mathbb{Q}^c,A+t\subset \mathbb{Q}^c,  tA\subset \mathbb{Q}^c,\dfrac{A}{t}\subset  \mathbb{Q}^c.$$
Throughout the paper, we essentially analyze the following problem.
Let $$\beta\in (1, \infty)$$ be a real, and $r$ be a rational number.
Then when do we have that  the following set
$$\{\beta^kr\mod 1:k\geq 0\}$$ takes only finitely many values. More generally, how can we find the non-trivial lower and upper bounds of the above set.
This is related to the theory of  sequences modulo one, see  \cite{Dubickas,Lagarias} and references therein.   For the above problem, it is natural to consider some Pisot numbers and rational numbers.
In \cite{Dubickas}, Dubickas proved that when $\beta$ is a Pisot number (an algebraic number is called a Pisot number if  all of whose other Galois conjugates have modulus strictly less than $1$) and $r$ is rational, then
$$\{\beta^kr\mod 1:k\geq 0\}$$ has finitely many limit points.

In \cite{Lagarias},   Flatto, Lagarias, and Pollington considered when $\beta$ is a rational number, they give the
 lower bound of the length of $$\{\beta^kr\mod 1:k\geq 0\}.$$ All these results are useful to our analysis on the translations of the  invariant sets and self-similar sets.

\section*{Acknowledgements}
 This work is
supported by K.C. Wong Magna Fund in Ningbo University.
This work is also supported by  Zhejiang Provincial Natural Science Foundation of China with
No.LY20A010009. The author would like to thank Simon Baker and Derong Kong for some   discussions on the uniform distribution of sequences.


\begin{thebibliography}{1}
\bibitem{Allouche}
Jean Paul, Allouche  and Michel, Cosnard.
\newblock The Komornik-Loreti Constant is Transcendental.
\newblock {\em Amer. Math. Monthly}, 107(5):448--449, 2000.


\bibitem{RSK}
Rafael Alcaraz~Barrera, Simon Baker, and Derong Kong.
\newblock Entropy, topological transitivity, and dimensional properties of
  unique {$q$}-expansions.
\newblock {\em Trans. Amer. Math. Soc.}, 371(5):3209--3258, 2019.





\bibitem{KarmaMartijn}
Karma Dajani and Martijn de~Vries.
\newblock Invariant densities for random {$\beta$}-expansions.
\newblock {\em J. Eur. Math. Soc. (JEMS)}, 9(1):157--176, 2007.


\bibitem{MK}
Martijn de~Vries and Vilmos Komornik.
\newblock Unique expansions of real numbers.
\newblock {\em Adv. Math.}, 221(2):390--427, 2009.


\bibitem{Dubickas}
Art\={u}ras Dubickas.
\newblock On the limit points of the fractional parts of powers of {P}isot
  numbers.
\newblock {\em Arch. Math. (Brno)}, 42(2):151--158, 2006.




\bibitem{EJK}
P{\'a}ul Erd\H{o}s, Istv{\'a}n Jo{\'o}, and Vilmos Komornik.
\newblock Characterization of the unique expansions
  {$1=\sum^\infty_{i=1}q^{-n_i}$} and related problems.
\newblock {\em Bull. Soc. Math. France}, 118(3):377--390, 1990.




\bibitem{Lagarias}
Leopold Flatto, Jeffrey~C. Lagarias, and Andrew~D. Pollington.
\newblock On the range of fractional parts {$\{\xi(p/q)^n\}$}.
\newblock {\em Acta Arith.}, 70(2):125--147, 1995.


\bibitem{GS}
Paul Glendinning and Nikita Sidorov.
\newblock Unique representations of real numbers in non-integer bases.
\newblock {\em Math. Res. Lett.}, 8(4):535--543, 2001.

\bibitem{Glendinning-Sidorov}
 Paul Glendinning and Nikita Sidorov.
\newblock The doubling map with asymmetrical holes.
\newblock {\em Ergodic Theory and  Dynamical Systems}, 35(4):1208–1228, 2015.

















\bibitem{Hutchinson}
John  Hutchinson.
\newblock Fractals and self-similarity.
\newblock {\em Indiana Univ. Math. J.}, 30(5):713--747, 1981.


\bibitem{Jia2021}
Qi Jia, Yuanyuan Li and Kan Jiang.
\newblock Irrational self-similar sets.
\newblock {\em To apper in Publ. Math. Debrecen.}, 2021.


\bibitem{Komornik}
Komornik, Vilmos and Loreti, Paola.
\newblock Unique developments in non-integer bases.
\newblock {\em Amer. Math. Monthly}, 105(7):636--639, 1998.

\bibitem{Kuipers1974}
L.~Kuipers and H.~Niederreiter.
\newblock {\em Uniform distribution of sequences}.
\newblock Pure and Applied Mathematics. Wiley-Interscience [John Wiley \&
  Sons], New York-London-Sydney, 1974.


\bibitem{LM}
Douglas Lind and Brian Marcus.
\newblock {\em An introduction to symbolic dynamics and coding}.
\newblock Cambridge University Press, Cambridge, 1995.

\bibitem{Mahler}
Kurt Mahler.
\newblock An unsolved problem on the powers of $3/2$.
\newblock {\em J. Austral. Math. Soc.}, 8,
313–321,
1968.


\bibitem{Mahler1984}
Kurt Mahler.
\newblock Some suggestions for further research.
\newblock {\em Bull. Austral. Math. Soc.}, 29(1):101--108, 1984.

\bibitem{SK}
Klaus Schmidt.
\newblock On periodic expansions of {P}isot numbers and {S}alem numbers.
\newblock {\em Bull. London Math. Soc.}, 12(4):269--278, 1980.




\bibitem{Nikita Sidorov}
Nikita Sidorov.
\newblock Almost every number has a continuum of $\beta$-expansions.
\newblock {\em Amer.
Math. Monthly}, 110(9):838–842, 2003.







\end{thebibliography}

\end{document}